\documentclass[submission]{tempclass}
\usepackage{geometry}                
\usepackage[latin1]{inputenc}
\usepackage{subfigure}

\usepackage{amssymb}
\usepackage{epstopdf}
\usepackage[square,sort,comma,numbers]{natbib}

\usepackage{amsmath}
\usepackage{amscd}
\usepackage{mathrsfs}
\usepackage{amscd}

\usepackage{setspace}
\newtheorem{thm}{Theorem}
\newtheorem{lem}[thm]{Lemma}
\newtheorem{cor}[thm]{Corollary}
\newtheorem{prop}[thm]{Proposition}
\newtheorem{defn}[thm]{Definition}

\renewcommand{\Re}{\mathbb{R}}

\newcommand{\dd}{\delta_{\mathbf{diam}}}

\renewcommand{\Re}{{\mathbb{R}}}

\author{David Bryant\addressmark{1}\thanks{Email: \email{david.bryant@otago.ac.nz}}
\and Paul F. Tupper\addressmark{2}\thanks{Email: \email{pft3@math.sfu.ca}}}
\title{Constant distortion embeddings of Symmetric Diversities}
\address{\addressmark{1}Dept.\ of Mathematics and Statistics, University of Otago, Dunedin 9054, New Zealand.\\
\addressmark{2}Dept.\ of Mathematics, Simon Fraser University. 8888 University Drive, Burnaby, British Columbia V5A 1S6, Canada.}


\received{\today}
\revised{\today}

\newcommand{\Pf}{\mathcal{P}_{\hspace{-0.2mm}\mathrm{fin}}   }

\newcommand{\hdelta}{{\widehat{\delta}}}

\newcommand{\diam}{\mathrm{diam}}

\begin{document}
\maketitle

\begin{abstract}
Diversities are like metric spaces, except that every finite subset, instead of just every pair of points, is assigned a value. Just as there is a theory of  minimal distortion embeddings of finite metric spaces into $L_1$, there is a similar, yet undeveloped, theory for embedding finite diversities into the diversity analogue of $L_1$ spaces. In the metric case, it is well known that an $n$-point metric space can be embedded into $L_1$ with  $\mathcal{O}(\log n)$ distortion. For diversities, the optimal distortion is unknown. Here, we establish the surprising result that \emph{symmetric} diversities, those in which the diversity (value) assigned to a set depends only on its cardinality, can be embedded in $L_1$ with constant distortion.
\end{abstract}

\section{Introduction}

Diversities are an extension of the concept of a metric space in which a non-negative value is assigned to every finite set of points, instead of just to pairs. 
 Formally, a {\em diversity} is a pair $(X,\delta)$ where $X$ is a set and $\delta$ is a function from the finite subsets of $X$ to $\Re$ satisfying  
\begin{quotation}
\noindent (D1) $\delta(A) \geq 0$, and $\delta(A) = 0$ if 
and only if 
$|A|\leq 1$. \\~\\
(D2) If $B \neq \emptyset$ then $\delta(A\cup B) + \delta(B \cup C) \geq \delta(A \cup C)$,
\end{quotation}
for all finite $A, B, C \subseteq X$. There is a clear correspondence between (D1), (D2) and the axioms for a metric space. In fact, if we define $d(a,b) = \delta(\{a,b\})$ for all $a,b$ we have that $(X,d)$ is a metric space, called the {\em induced metric} for $(X,\delta)$ \cite{BryantETAL2012a}. 

\begin{figure}[hbt]
\centerline{\includegraphics[width=0.9\textwidth]{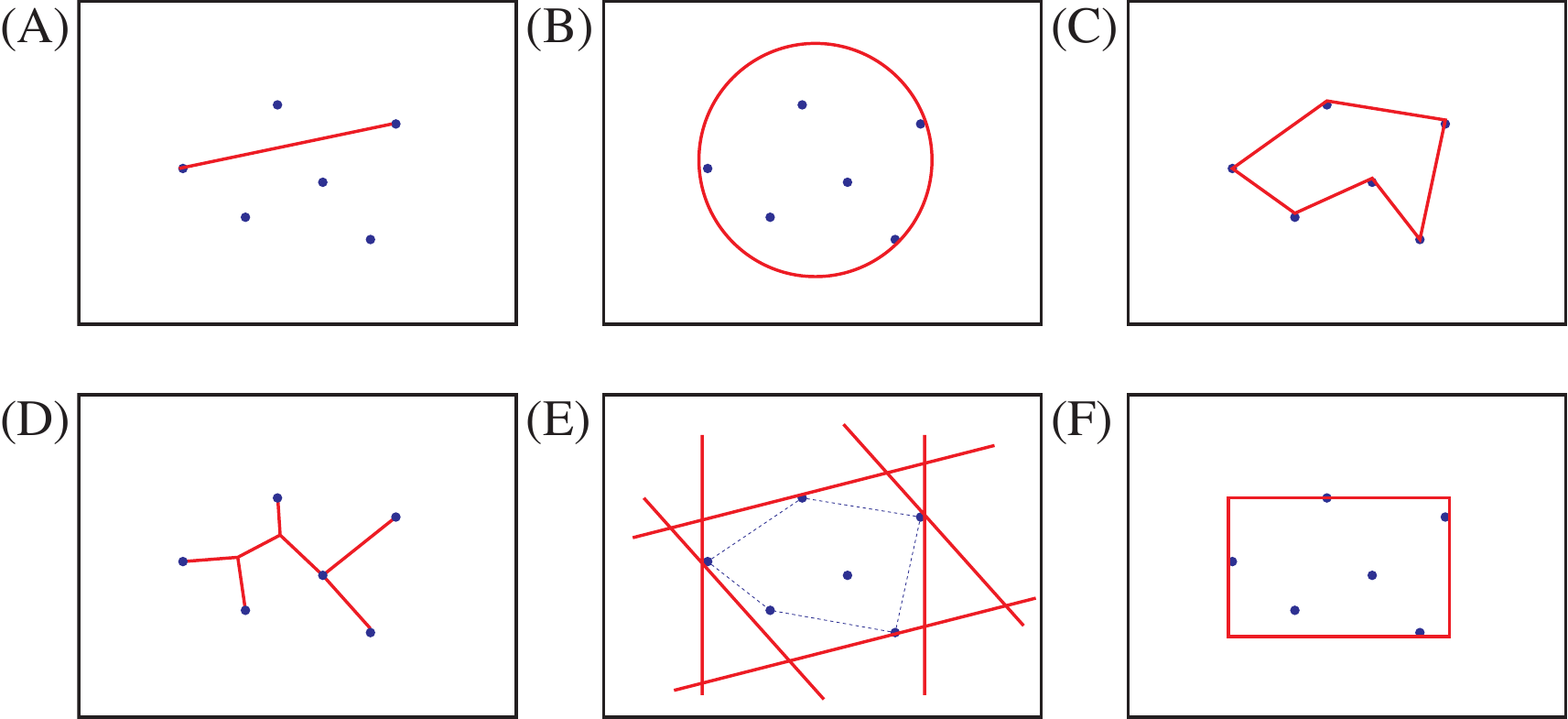}}
  \caption{\label{fig:examples}Six examples of diversities. Points indicate a finite subset $A$ of $R^n$. The diameter diversity (A) equals the maximum distance between points in $A$; The ball diversity (B) equals the diameter of the smallest ball enclosing $A$; The TSP diversity (C) equals half the length of the shortest tour through $A$; The Steiner diversity (D) equals the length of a minimal Steiner tree for $A$; The mean width diversity equals the scaled mean width of the convex closure of $A$; The $L_1$ diversity (E) equals the sum of the dimensions of the smallest axis-aligned box enclosing $A$.}
\end{figure}

Figure~\ref{fig:examples} illustrates several examples of diversities defined on $X = \Re^n$. The simplest is the {\em diameter diversity}. For each finite subset $A \subseteq X$, 
\[
\dd (A) = \max_{a,b \in A} \|a-b\|_2 = \diam(A).
\]
The {\em ball diversity}  equals the diameter of the smallest ball containing $A$. 
The {\em TSP diversity} equals half the length of the shortest tour visiting all the points in $A$. The {\em Steiner diversity} equals the length of the shortest Steiner tree connecting points in $A$. The {\em mean-width diversity} equals the (scaled) mean-width of the convex hull of $A$:
\[\delta_w(A) = \frac{\pi}{\mu_n(S^{n})} \int_{S^{n-1}} \left(\max_{a \in A} a \cdot u  - \min_{a \in A} a \cdot u \right) \, d\mu_{n-1}(u),\]
where $\mu_{n-1}$ denotes the surface measure on the unit sphere $S^{n-1}$ \cite{Taylor06}.
All of these diversities have the Euclidean metric as their induced metric. This is not the case, for example, for the {\em $\ell_1$ diversity}, defined on $\Re^n$ by 
\[\delta_1(A) = \sum_{i=1}^n \max\{|a_i - b_i|:a,b \in A\}.\]
The induced metric for the $\ell_1$ diversity is the $\ell_1$ metric. 
See \cite{BryantETAL2012a,BryantETAL2013a}, and below, for further examples of diversities.

 Diversities were first introduced by \cite{BryantETAL2012a} and it quickly became apparent that 
the concept leads to a rich and useful new area of mathematical theory and applications. 
Remarkable analogues have arisen between the non-linear analysis of metric spaces and diversity theory \cite{KirkETAL2014a,Piatek2014a,EspinolaETAL2014a} with a new and more general fixed point theorem for non-expansive maps being established by Esp\'inola and Pi\c{a}tek \cite{espinolaPiatek,KirkETAL2014a}. Diversity theory has led to new work in topology \cite{Poelstra2013b} and model theory 
\cite{BryantETAL2015a}. Diversities have also arisen in application areas ranging from evolutionary biology \cite{Steel2014a} to image recognition \cite{KumarETAL2015a}.

One part of the theory of metric spaces which has had a major impact on combinatorial optimization is low distortion embeddings of finite metrics. Linial et al.\ \cite{Linial95} showed how to use the mathematics of metric embeddings to help solve difficult problems in combinatorial optimization. The approach  inspired a large body of further work on metric embeddings and  their applications. In an earlier paper \cite{BryantETAL2013a} we showed that much of this theory translated over to diversities, promising an even larger toolbox of algorithmic techniques. Just as metric embedding provides a geometry of graphs \cite{Linial95},  diversity embeddings provides a geometry of  hypergraphs \cite{BryantETAL2013a}.

Many open problems remain. In particular, it is not known whether there exists a diversity version of Bourgain's celebrated embedding theorem: for every finite metric $(X,d)$, $|X| = n$, there is an embedding $\phi$ into $\ell_1$ metric space $(\Re^k,d_1)$ such that 
\[d(x,y) \leq d_1(\phi(x),\phi(y)) \leq D \cdot d(x,y),\]
where $k$ and $D$ are both $O(\log n)$. The minimum value $D$ for which these bounds hold is called the {\em distortion} of the embedding. It was shown in \cite{BryantETAL2013a} that there is no analogue of Bourgain's theorem satisfying both the distortion bound and the dimension bound simultaneously. The same paper gives polylogarithmic distortion bounds for embeddings from a wide range of diversities into $\ell_1$, though no general upper bound is known. See below for a formal presentation of this problem.

The current paper was motivated by a search for provable {\em lower} bounds for diversity embeddings. Many of the existing distortion bounds have been for diversities defined directly from the induced metric, such as the Steiner diversity or diameter diversity. We investigate the class of {\em symmetric diversities} which are not determined by their induced metrics and yet are straightforward to characterize. A diversity $(X,\delta)$ is symmetric if  it assigns the same value to sets of the same cardinality. These diversities are analogous to {\em equilateral sets} in metric geometry \cite{AlonETAL2003a,Matousek2010a}. Our main result is almost the complete opposite of what was expected: we show (Theorem~\ref{thm:EmbedSymm}) that symmetric diversities can be embedded in $\ell_1$ with {\em constant} distortion (albeit with large dimension). Rather than a source of lower bounds, these class might prove to be a building block for general upper bounds.


\section{Characterization of Symmetric Diversities}

\begin{defn}
Let $\Pf(X)$ denote the set of finite subsets of $X$. 	A diversity $(X,\delta)$ is {\em symmetric} if $\delta(A) = \delta(A')$ for all $A,A' \in \Pf(X)$ such that $|A| = |A'|$.
\end{defn}
A direct consequence of this definition is that symmetric diversities are determined by the values assigned to sets of each cardinality. Our first observation is that symmetric diversities correspond to (non-decreasing) sub-additive functions.
\begin{prop} \label{symAdditive}
Suppose that $(X,\delta)$ satisfies $\delta(\emptyset)=0$ and $\delta(A) = f(|A|-1)$ for all nonempty $A \in \Pf(X)$, where $f$ is some real-valued function on the positive integers. Then $(X,\delta)$ is a diversity if and only if
\begin{enumerate}
\item[(S1)] $f(0)=0$ and $f(j)>0$ for all $j>0$.
\item[(S2)] $f$ is non-decreasing.
\item[(S3)] $f(j + k ) \leq f(j) + f(k)$ for all $j,k \geq 0$, $j+k < |X|$.
\end{enumerate}
\end{prop}
\begin{proof}
Suppose $f$ satisfies (S1), (S2), (S3). Then $\delta(A) = 0$ when $|A| \leq 1$ and $\delta(A)>0$ when $|A| > 1$, by (S1). Suppose that $B \neq \emptyset$. From a Venn diagram, we have
\[|A \cup C|  \leq |A \cup B| + |B \cup C| - |B| \leq |A \cup B| + |B \cup C| - 1,\]
so by (S2) and (S3) we obtain
\begin{align*}
\delta(A \cup C) & =  f(|A \cup C| - 1) \\
&\leq  f(|A \cup B| - 1) + f(|B \cup C| - 1) \\
& =  \delta(A \cup B) + \delta(B \cup C).
\end{align*}
For the converse,  suppose that $(X,\delta)$ is a diversity. (S1), (S2) follow immediately. For (S3), let $A$ be a set of size $j+1$ and $B$ be a set of size $k+1$ such that $A$ and $B$ intersect in one point. Then $|A \cup B|= j+k+1$ and
\[
f(j+k) = \delta( | A \cup B| ) \leq \delta( A ) + \delta(B) = f(j) + f(k).
\]
\end{proof}
	
Our interest lies in embedding symmetric diversities, and embedding into symmetric diversities. We  define embeddings and distortion for diversities in the same way as for metric spaces. 

\begin{defn}\cite{BryantETAL2013a}
	Let $(X_1,\delta_1)$ and $(X_2,\delta_2)$ be two diversities and suppose $c \geq 1$. A map $\phi:X_1 \rightarrow X_2$ has {\em distortion} $c$ if there is $c_1,c_2 > 0$ such that $c = c_1 c_2$ and
\[\frac{1}{c_1} \delta_1(A) \leq \delta_2(\phi(A)) \leq c_2 \delta_1(A)\]
for all finite $A \subseteq X_1$. We say that $\phi$ is an {\em isomorphic embedding} if it has distortion $1$ and an {\em approximate embedding} otherwise. 
\end{defn}

We obtain tight bounds for embedding a general diversity into a symmetric diversity.
\begin{defn}
	Let $(X,\delta)$ be a diversity.
	 Define the {\em skewness} of $(X,\delta)$ by
	\[\gamma = \max\left\{\frac{\delta(A)}{\delta(B)} : |A| = |B| > 1 \right\}.\]
\end{defn}

\begin{prop} \label{skew}
	Let $(X,\delta)$ be a finite diversity with skewness $\gamma$. Then there is a embedding of $(X,\delta)$ into a symmetric diversity with distortion $\gamma$ and this bound is tight. 
\end{prop}
\begin{proof}
For each $k=0,1,\ldots,|X|-1$ define 
\[f(k) = \max\{\delta(A):A \subseteq X \mbox{ and } |A| = k+1 \}.\]
Then $f(0)= 0$ and $f$ is non-decreasing. Suppose that $0\leq j,k$ and $j+k<|X|$. There is $Y \subseteq X$ such that $|Y| = j+k+1$ and $f(j+k) = \delta(Y)$. Let $J$ and $K$ be disjoint subsets of $Y$ of cardinality $j$ and $k$, and let $y$ be the unique element in $Y - (J \cup K)$. Then 
\[f(j+k) = \delta(Y) \leq \delta(J \cup \{y\}) + \delta(K \cup \{y\}) \leq f(j) + f(k).\]
By Proposition~\ref{symAdditive}, $(X,\hdelta)$ defined by $\hdelta(A) = f(|A| - 1)$ is a symmetric diversity. Furthermore, for all $A \in \Pf(X)$ with $|A|>1$ we have 
\[\delta(A) \leq f(|A|-1) = \hdelta(A) \leq \gamma \delta(A).\]

To show that there is no embedding $\phi$ into a symmetric diversity $\delta_s$ with smaller distortion, consider two $A,B \in \Pf(X)$ such that $\gamma = \delta(A)/\delta(B)$ and $|A| = |B|$. 
Let $\delta_s$ be a symmetric diversity and let $c_1$, $c_2$ be constants such that
\[\frac{1}{c_1} \delta(A) \leq \delta_s(A) \leq c_2 \delta(A) \ \ \ \ 
\mbox{   and   } \ \ \ \
\frac{1}{c_1} \delta(B) \leq \delta_s(B) \leq c_2 \delta(B).\]
Since $\delta_s(A)=\delta_s(B)$ we get
\[
c_1 c_2 \geq \frac{\delta(A)}{\delta_s(A)} \frac{\delta_s(B)}{\delta(B)} = \gamma,
\]
as required.
\end{proof}

To conclude this section we derive a finite set of basis functions which can be used to approximate symmetric diversities up to a constant distortion. Recall that a function $f$ defined on some subset of the integers is concave if for all $a \leq k \leq b$ where $f$ is defined we have
\[
f(k) \geq  \frac{b-k}{b-a} f(a) + \frac{k-a}{b-a} f(b) .
\]
Not all functions $f$ coming from a symmetric diversity as in Proposition~\ref{symAdditive} are concave on $0, \ldots, |X|-1$. As an example, let $f(0)=0$, $f(1)=f(2)=1$, $f(k)=2$ for $k> 2$. Then $f$ satisfies S1--S3 of Proposition~\ref{symAdditive}, but $f(2) < ( f(1) + f(3) )/2$, violating concavity. However, every such $f$ can be approximated up to a factor of $2$.

\begin{lem}\label{symConcave} 
Suppose that $(X,\delta)$ is a finite symmetric diversity with $|X|=n$, where $\delta(A) = f(|A| - 1)$ for all nonempty $A\in \Pf(X)$. Then there is a concave and non-decreasing function $g$ such that $f(k) \leq g(k) \leq 2 f(k)$ for $k=0,1, \ldots, n-1$.	
\end{lem}
\begin{proof}
Let $g$ be the smallest concave function greater than $f$ on $0, \ldots, n-1$. Then $g$ is non-decreasing since, if not, we can replace $g(k)$ by $\min\{g(k),g(n-1)\}$ for all $k$ and obtain a concave function which is also greater than $f$. Also, $g(0)=0$ for similar reasons.

For all $k=0,\ldots,n-1$ we have $f(k) \leq g(k)$. Suppose $f(k) < g(k)$ for some $k$. Then by minimality there is an inequality of the form
\[
g(k) \geq \frac{b-k}{b-a} f(a) + \frac{k-a}{b-a} f(b),
\]
which holds as an equality, where $0 \leq a < k < b$. Now 
\[
f(b) = f \left( \frac{b}{k} k \right) \leq f \left( \left\lceil \frac{b}{k} \right\rceil  k \right) \leq \left\lceil \frac{b}{k} \right\rceil f(k) \leq \left(\frac{b}{k}+1\right) f(k),
\]
and $f(a) \leq f(k)$.
So 
\begin{eqnarray*}
g(k) & \leq &  \frac{b-k}{b-a}  f(k) + \frac{k-a}{b-a}   \left( \frac{b}{k} +1 \right) f(k)  \\
& = & f(k) +  \frac{k-a}{b-a}  \frac{b}{k} f(k) \leq 2 f(k),
\end{eqnarray*}
where the last inequality follows from 
\[
 \frac{k-a}{b-a}  \frac{b}{k}  =   \frac{1-a/k}{1-a/b} \leq 1,
\]
since $b>k$.
\end{proof}

\begin{thm}\label{symbasis}
Let $n>1$ be given.
Define the functions $\psi_1,\psi_2,\ldots,\psi_{n-1}$ by 
\[\psi_{i}(j) = \min(i,j),\]
for $j=0,1, \ldots, n-1$.
Let $(X,\delta)$ be a symmetric diversity with $|X| = n$. 
Then there are non-negative coefficients $\lambda_1,\ldots,\lambda_{n-1}$ such that for all nonempty $A \subseteq X$,
\[\delta(A) \leq \sum_{i=1}^{n-1} \lambda_i \psi_i(|A|-1) \leq 2 \delta(A).\]
\end{thm}
\begin{proof}
Let $f(j)=\delta(A)$ where $|A|=j+1$ for $j=0,\ldots,n-1$, as in Proposition~\ref{symAdditive}.
	By Lemma~\ref{symConcave} there is a concave and non-decreasing function $g:\{0,\ldots,n-1\}\rightarrow \Re$ such that $f(i) \leq g(i) \leq 2f(i)$ for all $i=0,\ldots,n-1$. For each $i=1,\ldots,n-2$ define $\lambda_i = 2g(i) - g(i+1) - g(i-1)$ and let $\lambda_{n-1} = g(n-1) - g(n-2)$. Since $g$ is concave and non-decreasing, all $\lambda_i$ terms are non-negative.
Then for all $1 \leq j \leq n-1$ we have
	\begin{eqnarray*} 
\sum_{i=1}^{n-1} \lambda_i \psi_i(j) & = & \sum_{i=1}^j \lambda_i \cdot i + \sum_{i=j+1}^{n-1} \lambda_i \cdot j \\
& = & \sum_{i=1}^j (2g(i) - g(i-1) - g(i+1)) \cdot i + j \cdot \sum_{i=j+1}^{n-2} (2g(i) - g(i-1) - g(i+1))   \\ && \quad  + \, j \cdot (g(n-1) - g(n-2)) \\
& = & g(j).
\end{eqnarray*}
	\end{proof}

We note that the functions $\psi_i$, $i = 1, \ldots, n-1$, each correspond to a diversity, since they satisfy the conditions of Proposition~\ref{symAdditive}.

\section{$L_1$ embedding of symmetric diversities}

In this section we prove our main result, that any finite, symmetric diversity can be embedded in an $L_1$ diversity with constant distortion. By Theorem~\ref{symbasis} every symmetric diversity can be approximated by a non-negative linear combination of functions $\psi_1,\ldots,\psi_{n-1}$ with at worst constant distortion. Hence to prove the main result, we need to show that each function $\psi_i$ can be embedded with constant distortion. First, we characterize  diversities on finite sets which are both symmetric and isomorphically embeddable  in $L_1$, that is, embeddable with constant 1. 

Given $U \subseteq X$ define the diversity $\delta_U$ by
\[\delta_U(A) = \begin{cases} 1, & \mbox{ if $A \cap U$ and $A \setminus U$ both non-empty}; \\ 0, & \mbox{ otherwise.} \end{cases} \]
In other words, $\delta_U(A) = 1$ when $U$ cuts $A$ into two parts. We say that $\delta_U$ is the {\em split diversity} for the split (bipartition) $U|(X-U)$. From \cite{BryantETAL2013a}
(see also \cite{BryantETAL2012b}) we have that $(X,\delta)$ is an $L_1$ embeddable diversity if and only if it is a non-negative linear combination of split diversities. 

\begin{prop} \label{prop:symL1}
Let $(X,\delta)$ be a finite, symmetric diversity and suppose $\delta(A) = f(|A| - 1)$ for all nonempty $A \subseteq X$. Then $\delta$ is $L_1$-embeddable if and only if there are non-negative $\lambda_1,\ldots,\lambda_{n-1}$ such that 
\begin{equation}
f(k) = \sum_{\ell = 1}^{n-1} \lambda_\ell \left(\binom{n}{\ell} - \binom{n-k-1}{n - \ell} - \binom{n-k-1}{\ell} \right). 
\end{equation}
for all $k=1,\ldots,n-1$.
\end{prop}
\begin{proof}
Suppose that $\delta$ is $L_1$-embeddable.  Hence there there there are non-negative weights $w_B$ for all $B \subseteq X$ such that 
$w_B = w_{(X-B)}$, $w_X = 0$, and for all $A \subseteq X$
\begin{eqnarray*}
\delta(A) &=& \sum_{B \subseteq X} w_B \delta_B(A). 
\end{eqnarray*}
From Theorem 4 of \cite{BryantETAL2012b} we have
\begin{eqnarray*}
w_B &=& \frac{1}{2} \sum_{A : B \subseteq A} (-1)^{|A|+|B|+1} \delta(A) \\
	 &=& \frac{1}{2} \sum_{A : B \subseteq A} (-1)^{|A|+|B|+1} f(|A| - 1) \\ 
	& = & \frac{1}{2} \sum_{k=|B|}^n \binom{n-|B|}{k-|B|} (-1)^{k+|B|+1} f(k-1).
\end{eqnarray*}
Hence $w_B$ depends only on the cardinality of $B$, not  on the elements of the set. That is, $w_B = w_{B'}$ whenever $|B'| = |B|$. Define $\lambda$ so that $\lambda_{\ell} = w_B$ for all $B$ with $|B| = \ell$. Then 
\begin{eqnarray*}
\delta(A) &=& \sum_{B \subseteq X} w_B \delta_B(A) \\
& = & \sum_{\ell = 1}^{n-1} \sum_{B \subseteq X,\, |B| = \ell} \lambda_{\ell} \delta_B(A) \\
& = & \sum_{\ell = 1}^{n-1} \lambda_{\ell} \left| \{B:|B| = \ell,\, A \not \subseteq B,\, B\cap A \neq \emptyset\} \right| \\
& = & \sum_{\ell = 1}^{n-1} \lambda_{\ell} \left(\binom{n}{\ell} - \binom{n-|A|}{\ell - |A|} - \binom{n-|A|}{\ell}\right),
\end{eqnarray*}
so
\[f(k) = \sum_{\ell = 1}^{n-1} \lambda_{\ell} \left(\binom{n}{\ell} - \binom{n-k-1}{n-\ell} - \binom{n-k-1}{\ell}\right).\]

For the converse, suppose that coefficients $\lambda_\ell$ satisfy the conditions of the Proposition. Let $w_B = \lambda_{|B|}$ for all $B \subseteq X$ and the result follows.
\end{proof}

For each $\ell = 1,2,\ldots,n-1$ we define the function $\varphi_\ell$ for  $k= 0,1,\ldots,n-1$ by
\[\varphi_\ell(k) =  \frac{\binom{n}{\ell} - \binom{n-k-1}{n-\ell} - \binom{n-k-1}{\ell}}{2\binom{n-2}{\ell-1}}.
\] 
Figure~\ref{fig:plot_of_varphi} shows plots of $\varphi_\ell(k)$ versus $k$ for $n=20$ and $\ell=1,\ldots,10$.
 By Proposition~\ref{prop:symL1} any symmetric diversity $\delta$ with $\delta(A) = f(|A| - 1)$ is $L_1$-embeddable exactly when $f$ is a non-negative combination of the functions $\varphi_1,\ldots,\varphi_{n-1}$.\\

\begin{figure}
\centerline{
\includegraphics[width=10cm]{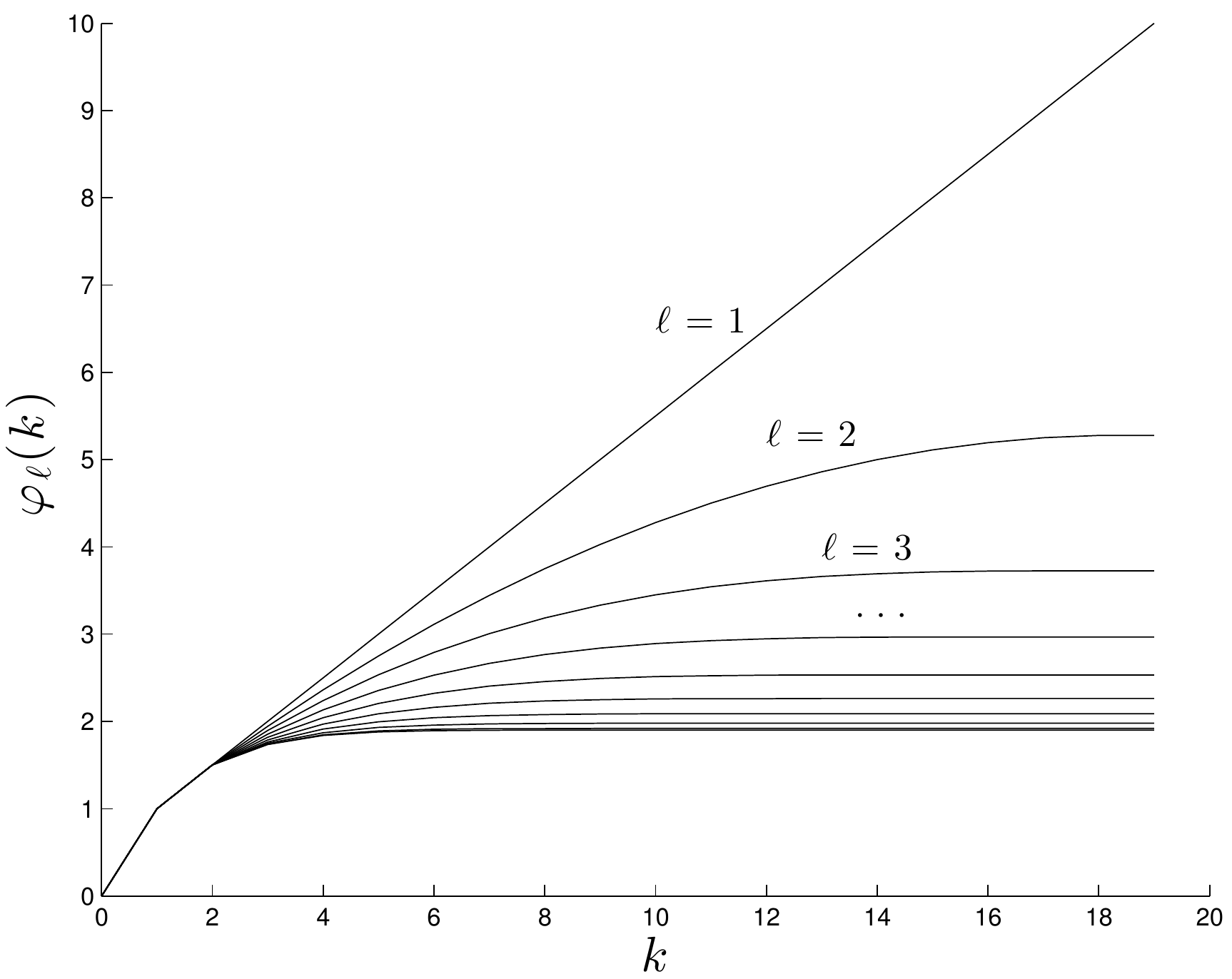}}
\caption{  \label{fig:plot_of_varphi} The values of  $\varphi_\ell(k)$ for  $\ell =1,\ldots, n/2$, for $n=20$. }
\end{figure}

\begin{prop} \label{varphi}
The functions $\varphi_\ell$ are concave and non-decreasing for $0 \leq k \leq n-1$, with $\varphi_\ell(0)=0$, $\varphi_\ell(1) = 1$ and
\[\varphi_\ell(n-1) = \frac{n(n-1)}{2\ell(n-\ell)}.\]
\end{prop}
\begin{proof}
By Proposition~\ref{symAdditive}, $\varphi_\ell$ is non-decreasing 	and sub-additive. To show concavity, note that (after some algebraic manipulation)
\begin{eqnarray*}
2 \binom{n-2}{\ell-1} \left[ \Big( \varphi_\ell(k+1) - \varphi_\ell(k)\Big)  - \Big( \varphi_\ell(k+2) - \varphi_\ell(k+1)\Big) \right] \hspace{-8cm} && \\
&=&\binom{n-k-3}{\ell} \frac{(\ell-1)^2}{(n-k-\ell-1)(n-k-\ell-2)} + \binom{n-k-3}{n-\ell} \frac{(n-\ell)(n-\ell-1)}{(\ell - k - 1)(\ell - k - 2)},
\end{eqnarray*}
which is non-negative.
\end{proof}

By Theorem~\ref{symbasis}, 
if $\delta$ is a symmetric diversity with $\delta(A) = f(|A| - 1)$ for all $A$ then there are coefficients 
$\lambda_1,\ldots,\lambda_{n-1}$ such that 
\[f(k) \leq \sum_{i=1}^{n-1} \lambda_i \psi_i(k) \leq 2 f(k),\]
for $k=0, \ldots, n-1$,
where $\psi_i(k) = \min\{i,k\}$. Hence the problem of embedding general symmetric diversities in $L_1$ with constant distortion reduces to the problem of embedding a diversity $\delta_i$,  defined by  $\delta_i(A) = \psi_i(|A| - 1)$ for all $A \subseteq X$. We do this in two steps.  First we approximate each function $\psi_i$ by a function $\Psi_{x(\ell)}$, with $\Psi_{x(\ell)}(k)=\min\{x(\ell),k\}$ 
where  
\[
 x(\ell) = \varphi_\ell(n-1) = \frac{n(n-1)}{2\ell(n-\ell)},\]
 for some $\ell$,  $1 \leq \ell \leq \lfloor n/2 \rfloor$. Second, we show that the function $\Psi_{x(\ell)}$ can itself be approximated (up to a constant scalar) by $\varphi_\ell$. In Figure~\ref{fig:plot_of_psi} we show plots of both $\psi_i$ and $\Psi_{x(\ell)}$ for the case of $n=20$.

\begin{figure}
\includegraphics[width=7.5cm]{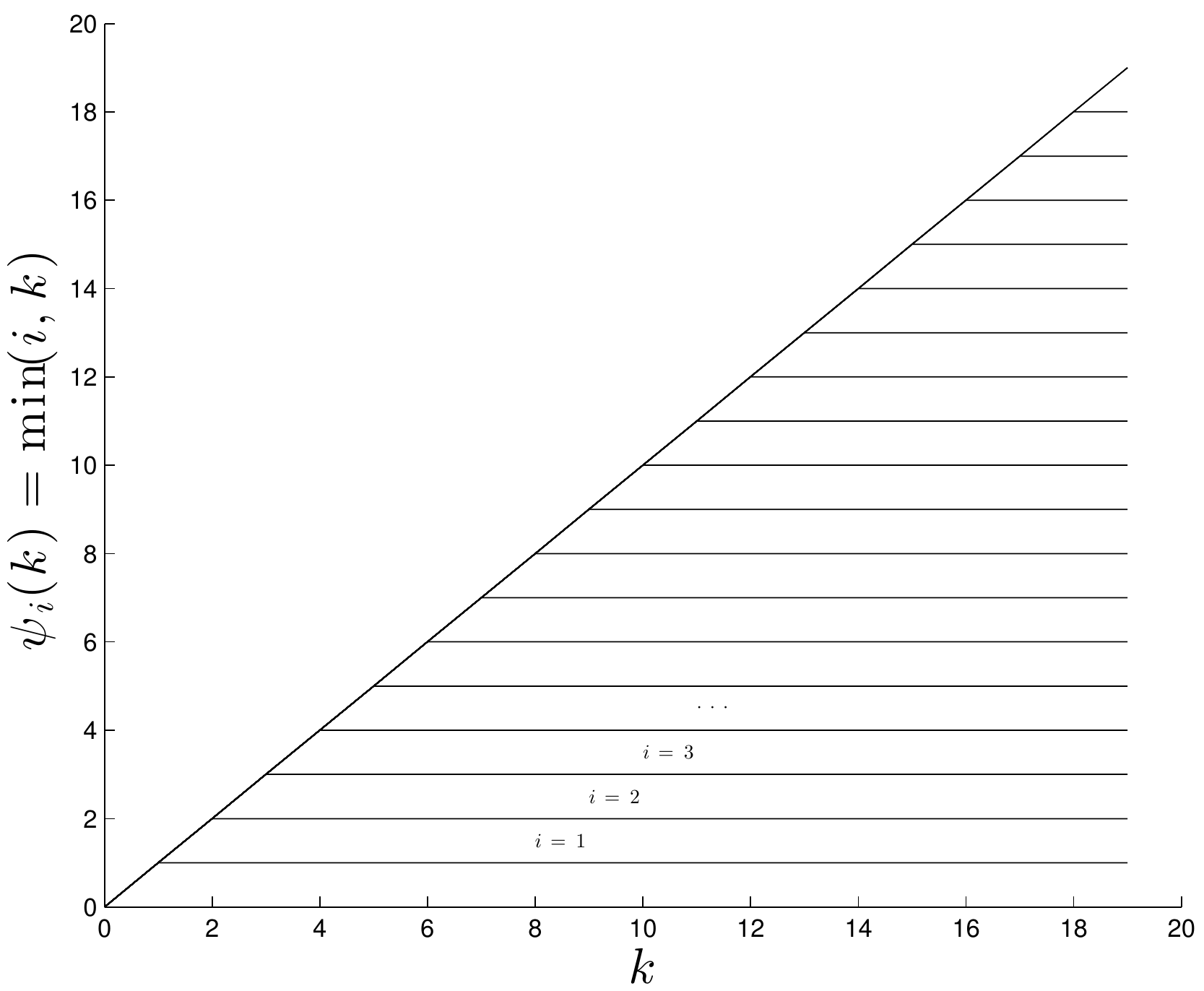}
\includegraphics[width=7.5cm]{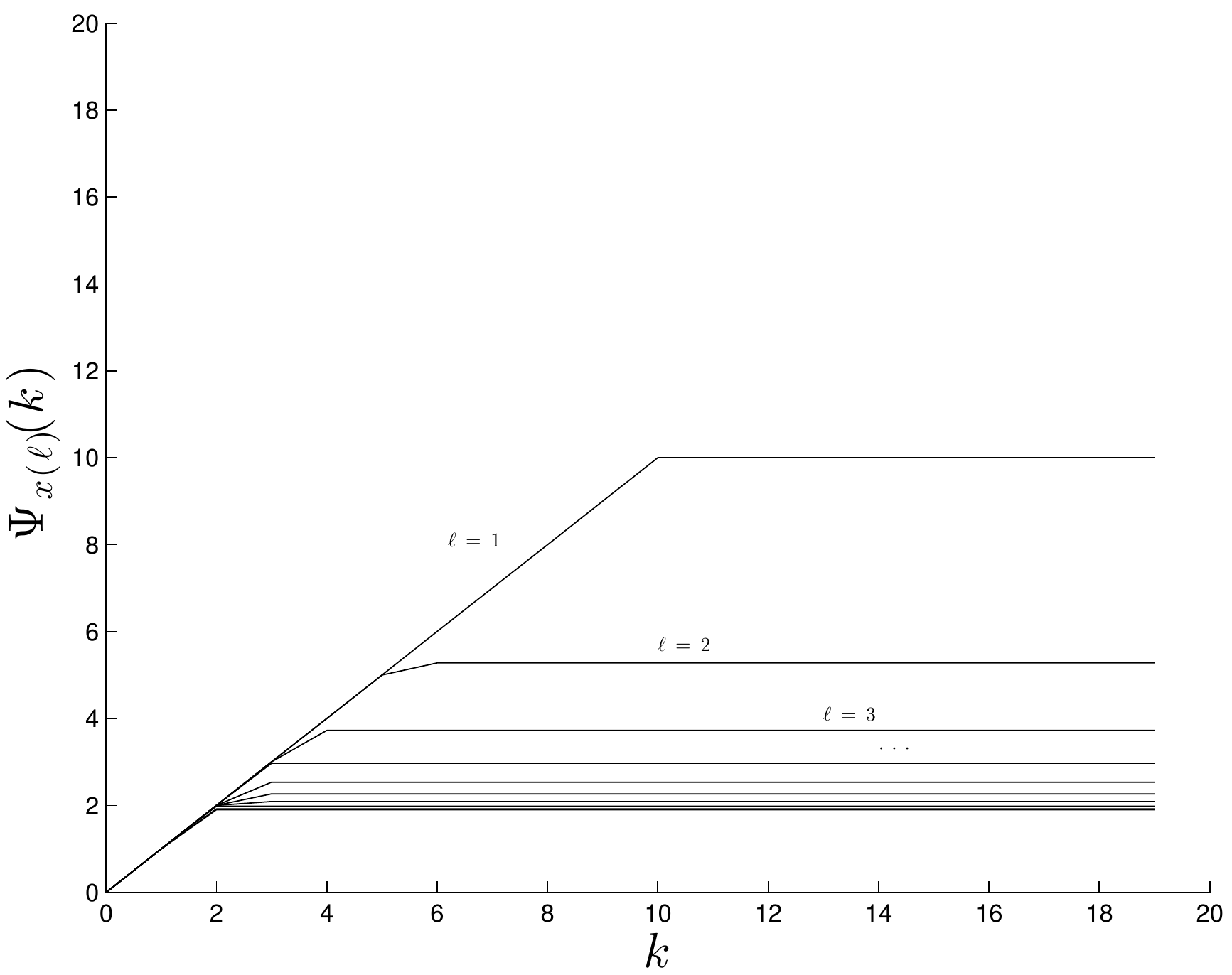}
\caption{  \label{fig:plot_of_psi} On the left, plots of $\psi_i(k)$ for $i=1,\ldots, n$. On the right, plots of $\Psi_{x(\ell)}(k)$ for $\ell=1,\ldots,n/2$. In both plots $n=20$. }
\end{figure}

\begin{lem} \label{approxPsi}
If $i=1$, letting $\ell = \lfloor n/2 \rfloor$ gives
\[\frac{1}{2} \Psi_{x(\ell)}(k) \leq \psi_i(k) \leq  \Psi_{x(\ell)}(k),\]
for $k=0, \ldots, n-1$.
If $i \geq 2$ then there is $\ell \in \{1,2,\ldots,\lfloor n/2 \rfloor\}$ such that
\[\Psi_{x(\ell)}(k) \leq \psi_i(k) \leq 2 \Psi_{x(\ell)}(k),\]
for all $k=0,\ldots,n-1$.
\end{lem}
\begin{proof}
First consider the case $i=1$. Note that $\psi_1(0)=0$ and $\psi_1(k) = 1$ for all $k=1,\ldots,n-1$. 
Letting $\ell = \lfloor n/2 \rfloor$ gives
\[
\Psi_{x(\ell)}(k)= \min\{    x(\ell) , k\} = \min \left\{   \frac{ n (n-1)}{ 2 \lfloor  n/2 \rfloor  \lceil n/2 \rceil}, k \right\}.
\]
We have the following inequalities:
\[
 1 \leq   \frac{ n (n-1)}{ 2 \lfloor  n/2 \rfloor  \lceil n/2 \rceil} \leq  2.
\]
So for $k=1,\ldots, n-1$, $1 \leq \Psi_{x(\ell)}(k) \leq 2$. Hence we get
\[
\frac{1}{2} \Psi_{x(\ell)}(k) \leq \psi_1(k) \leq \Psi_{x(\ell)}(k),
\]
for all $k$.

Now suppose $i \geq 2$. For fixed $n$, the function $x(\ell)$ has a maximum $x(1) = n/2$ and a minimum $x(\lfloor n/2 \rfloor) \leq 2$. For all $\ell = 1,2,\ldots,\lfloor n/2 -1 \rfloor$ we have
 \[1 \leq \frac{x(\ell)}{x(\ell+1)} = \left(1 + \frac{1}{\ell}\right)\left(1 - \frac{1}{n-\ell}\right) \leq 2.\]
 That is, the values $x(\ell)$ for $\ell = 1,\ldots, \lfloor n/2 \rfloor$ span from below $2$ up to $n/2$ and the ratio between successive $x(\ell)$ never exceeds $2$. 
 Hence for all $i \geq 2$  there is $\ell$ such that 
 \[x(\ell) \leq i \leq 2 x(\ell).\]

Using this $\ell$, when $1 \leq k \leq x(\ell)$
\[ \Psi_{x(\ell)}(k)=k = \psi_i(k).\]
When $x(\ell) \leq k \leq i$
\[\Psi_{x(\ell)}(k) = x(\ell)  \leq k = \psi_{i}(k) \leq i \leq 2 x(\ell) = 2\Psi_{x(\ell)}(k).\]
When $i \leq k \leq n-1$
\[\Psi_{x(\ell)}(k) = x(\ell) \leq \psi_{i}(k) = i \leq 2 x(\ell) = 2\Psi_{x(\ell)}(k).\]
 The result follows.
 \end{proof}

%
%

\begin{thm} \label{thm:EmbedSymm}
Every finite symmetric diversity can be embedded in $L_1$ with constant distortion.
\end{thm}
\begin{proof}
We prove that there is a constant $K$ independent of $n$ such that for each $\ell$ and all $k=0,1,\ldots,n-1$ we have
\[\varphi_{\ell}(k) \leq \Psi_{x(\ell)}(k) \leq K \varphi_{\ell}(k).\]
The result then follows from Theorem~\ref{symbasis}, Proposition~\ref{prop:symL1} and  Lemma~\ref{approxPsi}.

Recall that $x(\ell)= \frac{n(n-1)}{2 \ell (n-\ell)} \leq \frac{n}{2}$ for all $\ell$.

First consider the case that $x(\ell) \leq k \leq n-1$. For these values of $k$ the function $\Psi_{x(\ell)}(k) = x(\ell)$. Since $\varphi_{\ell}(k)$ is non-decreasing and $\varphi_{\ell}(n-1) = x(\ell)$, we have $\varphi_{\ell}(k) \leq x(\ell)=\Psi_{x(\ell)}(k)$ and the lower bound  on $\Psi_{x(\ell)}(k)$ is established.

For the bound in the other direction, it suffices to show that  $\varphi_\ell(k) /  \Psi_{x(\ell)}(k)$ is bounded away from zero for all $k$ and $\ell$. We have
\begin{eqnarray*}
\frac{ \varphi_\ell(k) }{ \Psi_{x(\ell)}(k)}& = &\frac{\varphi_{\ell}(k)}{x(\ell)} = \frac{\varphi_{\ell}(k)}{\varphi_{\ell}(n-1)} \\
& = &  \frac{\binom{n}{\ell} - \binom{n-k-1}{n-\ell} - \binom{n-k-1}{\ell}}{\binom{n}{\ell}} \\
& = & 1 - \frac{\ell (\ell-1) \cdots (\ell-k)}{n (n-1) \cdots (n-k)} - \frac{(n-\ell) (n-\ell-1) \cdots (n-\ell-k)}{n (n-1) \cdots (n-k)} \\
& \geq & 1 - \left(\frac{\ell}{n}\right)^{k+1} - \left(\frac{n-\ell}{n}\right)^{k+1}.
\end{eqnarray*}
Note that the quantity on the right is increasing with respect to $k$.
Fix $\kappa = 0.2$ and let $y = y(\ell) = \frac{n^2}{(\kappa + 2) \ell (n-\ell)}$. Then if $n > 1 + 2 / \kappa = 11$ we have $x(\ell) \geq y$ and so 
\begin{eqnarray*} 
\frac{ \varphi_\ell(k) }{ \Psi_{x(\ell)}(k)} &  \geq &  1 - \left(\frac{\ell}{n}\right)^{y+1} - \left(\frac{n-\ell}{n}\right)^{y+1},
\end{eqnarray*}
for all $k$, $x(\ell) \leq k \leq n-1$.
Let $z = \frac{\ell}{n}$, so $y = \frac{1}{(\kappa + 2)z(1-z)}$ and
\begin{eqnarray*}
\frac{ \varphi_\ell(k) }{ \Psi_{x(\ell)}(k)}& \geq & 1 - z^{\frac{1}{(\kappa + 2)z(1-z)}+1} - (1-z)^{\frac{1}{(\kappa + 2)z(1-z)}+1} \\
& = & 1 - f_1(z) - f_2(z),
\end{eqnarray*}
where
\begin{eqnarray*}
f_1(z) & = & \left(z\right)^{\frac{1}{(\kappa + 2)z(1-z)}+1}, \\
f_2(z) & = & \left(1-z\right)^{\frac{1}{(\kappa + 2)z(1-z)}+1}.
\end{eqnarray*}
The function $f_1(z) + f_2(z)$ is symmetric on the interval $(0,1)$. By taking derivatives we see that $f_1(z)$ is increasing on $(0,0.5]$, with a maximum  of $2^{-\frac{4}{\kappa+2} - 1}$ at $z = 0.5$. The function $f_2(z)$ is decreasing on $(0,0.5]$ with a supremum of  $e^{-\frac{1}{\kappa+2}}$ as $z \leftarrow 0$. Hence when $\kappa = 0.2$ we have $1 - f_1(z) - f_2(z) > 0.2$ and so 
\begin{equation}\label{eqn:pesky}
\frac{ \varphi_\ell(k) }{ \Psi_{x(\ell)}(k)}  \geq   1/5,
 \end{equation}
for all $x(\ell) \leq k \leq n-1$.

To complete the proof, we consider the $k$ such that $0 \leq k \leq x(\ell)$.  Recall that in this range, $\Psi_{x(\ell)}(k)=k$.
We have $\varphi_{\ell}(0)=0$ and $\varphi_{\ell}(1)=1$. Since $\varphi_\ell$ is concave, $\varphi_{\ell}(k) \leq k$ for all $k \geq 1$, establishing the upper bound on $\varphi_{\ell}$.
Furthermore, since 
\(\varphi_{\ell}(x(\ell)) \geq x(\ell)/5\)
by \eqref{eqn:pesky}
and $\varphi_{\ell}(k)$ is concave we have that the graph of $\varphi_{\ell}(k)$ on $k\in[0,x(\ell)]$ lies above the line between $(0,0)$ and $(x(\ell),x(\ell)/5)$. Hence
\[\frac{1}{5} \psi_{x(\ell)} \leq \varphi_{\ell}(x(\ell)) \leq \psi_{x(\ell)},\]
for all $0 \leq k \leq x(\ell)$, and therefore for all $0 \leq k \leq n-1$.
\end{proof}

As a direct corollary we obtain a general (if perhaps not very tight) bound of the distortion required to embed any diversity.
\begin{cor}
	Let $(X,\delta)$ be any finite diversity and let $\gamma = \max\{\frac{\delta(A')}{\delta(A)} : |A| = |A'| > 1\}$. Then $(X,\delta)$ can be embedded in an $L_1$ diversity with distortion at most $O(\gamma)$.
\end{cor}
\begin{proof}
By Proposition~\ref{skew}, $(X,\delta)$ can be embedded in a symmetric diversity with distortion $\gamma$, and by Theorem~\ref{thm:EmbedSymm} this symmetric diversity can be embedded in $L_1$ with constant distortion.
\end{proof}

%
%
%


\bibliographystyle{abbrvnat}
\bibliography{SymmetricDiversities}

\end{document}